\documentclass{amsart}

\usepackage{amsmath, amssymb, amsthm, color}
\usepackage[hidelinks,colorlinks=true,linkcolor=blue, citecolor=black,linktocpage=true]{hyperref}

\newtheorem{theorem}{Theorem}[section]

\newtheorem{proposition}[theorem]{Proposition}

\newtheorem{corollary}[theorem]{Corollary}

\newcommand\Z{\mathbb{Z}}

\newcommand\HF{\textit{HF}}

\newcommand{\HFh}{\widehat{\HF}}
\newcommand{\HFhat}{\HFh}

\newcommand{\HFp}{\HF^+}
\newcommand{\F}{\mathbb{F}}
\newcommand{\HFred}{\HF_\textup{red}}
\newcommand{\CFKi}{\mathit{CFK}^\infty}

\newcommand\co{\colon}
\newcommand\Cone{\operatorname{Cone}}
\newcommand\rk{\operatorname{rk}}
\newcommand\mfs{\mathfrak{s}}
\newcommand\mft{\mathfrak{t}}

\theoremstyle{definition}

\theoremstyle{remark}
\newtheorem{remark}[theorem]{Remark}

\title{Dehn surgery and non-separating two-spheres}
\date{}

\author[Jennifer Hom]{Jennifer Hom}
\thanks{The first author was partially supported by NSF grant DMS-1552285.}
\address{School of Mathematics, Georgia Institute of Technology, Atlanta, GA, USA}
\email{hom@math.gatech.edu}

\author[Tye Lidman]{Tye Lidman}
\thanks{The second author was partially supported by NSF grant DMS-1709702 and a Sloan Fellowship.}
\address{Department of Mathematics, North Carolina State University, Raleigh, NC 27607}
\email{tlid@math.ncsu.edu}

\begin{document}

\begin{abstract}  
When can surgery on a null-homologous knot $K$ in a rational homology sphere produce a non-separating sphere? We use Heegaard Floer homology to give sufficient conditions for $K$ to be unknotted. We also discuss some applications to homology cobordism, concordance, and Mazur manifolds.
\end{abstract}

\maketitle

\section{Introduction}
One of the most fundamental constructions in three-manifold topology is Dehn surgery.  By the theorems of Lickorish and Wallace, every closed, connected, oriented three-manifold is obtained by surgery on a link in $S^3$.  Additionally, 4-dimensional 2-handle attachments induce a cobordism from a three-manifold to the result of surgery.  It is therefore a fundamental question to understand the behavior of three-manifolds under Dehn surgery.  In this note, we focus on surgery on knots.  Two main questions are {\em geography} (which three-manifolds are obtained by surgery on a knot) and {\em botany} (which knots surger to a fixed three-manifold).  

For example, Gabai's ``Property R theorem'' \cite{Gabai:propR} shows that only 0-surgery on the unknot in $S^3$ can produce $S^2 \times S^1$.  The proof passes through taut foliations, and as a result, shows that 0-surgery on a non-trivial knot is not $S^2 \times S^1$ {\em and} is prime (i.e., the 0-surgery is irreducible), giving strong geography constraints.  Note that this implies that a four-manifold built with one 0-handle, one 1-handle, one 2-handle, and boundary $S^3$ is necessarily diffeomorphic to $B^4$.  Similarly, Gordon and Luecke's celebrated ``knot complement theorem'' \cite{GordonLuecke} answers the botany problem for surgeries from $S^3$ to $S^3$: only the unknot admits non-trivial $S^3$ surgeries. This shows that a closed four-manifold with one 0-handle, one 2-handle, and one 4-handle is necessarily diffeomorphic to $\mathbb{C}P^2$.   

In this article, we study a more general question: when can surgery on a knot in a three-manifold (other than $S^3$) produce an $S^2 \times S^1$ summand?  In previous work of Daemi, the second author, Vela-Vick, and Wong \cite{DLVVW}, some constraints were given on the geography problem.  Here, we answer both the botany and geography problems in several different settings.  While many of the arguments below are standard, we believe it is beneficial to the community for these results to be written down.   
   
We begin with a generalization of Property R to arbitrary rational homology spheres.  
\begin{theorem}\label{thm:main}
Let $Y$ be a rational homology sphere and $K$ a nullhomologous knot in $Y$.  Suppose $Y_0(K) = N \# S^2 \times S^1$.  If $\dim \HFhat(N) = \dim \HFhat(Y)$, then $N = Y$ and $K$ is unknotted.  Otherwise, $\dim \HFhat(N) < \dim \HFhat(Y)$.  
\end{theorem}

Theorem~\ref{thm:main} has a number of immediate applications.

\begin{corollary}
Let $K$ be a nullhomotopic knot in a prime rational homology sphere $Y$.  If $Y_0(K)$ contains a non-separating two-sphere, then $K$ is unknotted.  
\end{corollary}
\begin{proof}
It is shown in \cite[Theorem 1.8]{DLVVW} that under these hypotheses, $Y_0(K) = Y \# S^2 \times S^1$.  By Theorem~\ref{thm:main}, $K$ is unknotted.  
\end{proof}

\begin{corollary}\label{cor:trivial}
Let $Y$ be a rational homology sphere and let $W : Y \to Y$ be a rational homology cobordism with a handlebody decomposition with a total of two handles.  Then, $W$ is diffeomorphic to a product.  
\end{corollary}
\begin{proof}
Since $W$ is a rational homology cobordism, after possibly flipping $W$ upside down, $W$ consists of a single 2-handle and a single 3-handle.  Therefore, $Y$ has a surgery to $Y \# S^2 \times S^1$.  The result now follows from Theorem~\ref{thm:main}.  
\end{proof}

\begin{remark}
It seems reasonable to conjecture that a rational homology cobordism from a 3-manifold to itself without 3-handles is homeomorphic to a product.  It seems more ambitious, but still feasible, to believe that such a cobordism is diffeomorphic to a product.
\end{remark}

\begin{corollary}\label{cor:trivial-rank-1}
Suppose that there exists an integral homology cobordism $W$ from a rational homology sphere $Y$ to a three-manifold $Z$ consisting of a single 1-handle and a single 2-handle.  If $\dim \HFred(Z) = 1$, then $W$ is diffeomorphic to a product.  
\end{corollary}
\begin{proof}
By \cite[Theorem 1.19]{DLVVW}, $\dim \HFred(Y) = 0$ or 1.  If $\dim \HFred(Y) = 1$, then $\dim \HFhat(Y) = \dim \HFhat(Z)$, since $\dim \HFred = 1$ implies $\dim \HFhat = |H_1| + 2$ and $|H_1(Y)| = |H_1(Z)|$.  The result follows from Theorem~\ref{thm:main} by applying the arguments in Corollary~\ref{cor:trivial}.  (The fact that $W$ is an integral homology cobordism implies that the relevant surgery is along a nullhomologous knot.)  Next, suppose $\dim \HFred(Y) = 0$.  By the Spin$^c$-conjugation invariance of Heegaard Floer homology, we see that $\dim \HFred(Z, \mathfrak{s}) = 1$ in a self-conjugate Spin$^c$-structure $\mathfrak{s}$.  As shown by F. Lin in \cite{LinStein}, this implies that his correction terms $\alpha, \beta, \gamma$ are not all equal for $\mathfrak{s}$.  However, for an L-space, they are all equal.  This is a contradiction, since $\alpha, \beta, \gamma$ are preserved under integral homology cobordisms for each self-conjugate Spin$^c$ structure.  
\end{proof}

Note that the Brieskorn spheres $\Sigma(2,3,7)$ and $\Sigma(2,3,11)$ satisfy $\dim \HFred = 1$.  

Corollaries~\ref{cor:trivial} and \ref{cor:trivial-rank-1} can be seen as ``manifold versions'' of the following special case of a theorem of Gabai \cite[Theorem 1]{Gabai:superadditive}: a self-ribbon concordance with one minimum and one saddle is trivial.  (This was explained to us by Maggie Miller.)  In fact, one can recover a slight variant of this result using Theorem~\ref{thm:main}.  

\begin{corollary}
Let $K$ be a nullhomologous knot in a rational homology sphere $Y$.  Perform a band-sum with an unknot and denote the resulting knot by $K'$.  Suppose $K'$ is detected by its complement, which we additionally assume is irreducible and boundary irreducible.  If $\CFKi(K) \cong \CFKi(K')$, then $K'$ is isotopic to $K$ and the exterior of the resulting concordance is smoothly the trivial cobordism.  
\end{corollary}
Note that if $Y = S^3$ and $K$ is non-trivial, then the hypotheses apply for any $K'$ by \cite{Gordon:ribbon, GordonLuecke}.  For notation, we will write $E(X)$ to denote the exterior of the submanifold $X$.  (The ambient manifold will be clear from context.)
\begin{proof}
Let $C\co (Y,K) \to (Y,K')$ be the ribbon concordance in $Y \times I$ given by a single birth and saddle specified by the band-sum.  Since $K'$ is determined by its complement, it suffices to show that $E(C)$ is smoothly $E(K) \times I$.  

Note that $E(C)$ is an integer homology cobordism from $E(K)$ to $E(K')$ which consists of a single 1-handle and 2-handle addition.  Reversing orientation and flipping upside-down, we see that there exists a knot $J$ knot in $E(K')$ with an $E(K) \# S^2 \times S^1$ surgery.  Since $K$ and $K'$ are nullhomologous, we see that $J$ is necessarily nullhomologous in $E(K')$.  Note that if we can show that $J$ is trivial, then $E(C) = E(K) \times I$ and we are done.  

Write $J^{(n)}$ for the induced knot in $Y_{n}(K')$.  Then, 0-surgery on $J^{(n)}$ results in $Y_{n}(K) \# S^2 \times S^1$.  Since $\CFKi(K) \cong \CFKi(K')$, the large surgery formula of Ozsv\'ath-Szab\'o \cite[Theorem 4.4]{OS:knots} implies that $\dim \HFhat(Y_n(K)) = \dim \HFhat(Y_n(K'))$ for large $n$.  Therefore, by Theorem~\ref{thm:main}, $J^{(n)}$ is unknotted in $Y_n(K')$  for large $n$.  Since $Y_n(K')$ is not $S^3$ for large $n$, it follows that $E(J^{(n)}) = D^2 \times S^1 \# Y_n(K')$ is a reducible manifold for all large $n$.  

In other words, $E(K' \cup J)$ has infinitely many reducible fillings.  However, an irreducible, boundary-irreducible three-manifold with only toral boundary components has at most finitely many reducing fillings along a given boundary component (see for example \cite{Gordon:fillingsurvey}).  Therefore, $E(K' \cup J)$ is either boundary reducible or reducible.  Since $K'$ is non-trivial, if $E(K' \cup J)$ is boundary reducible, then the toral boundary component coming from $J$ must be the one that compresses, and we see that $J$ must be unknotted in the exterior of $K'$ completing the proof.  On the other hand, if $E(K' \cup J)$ is reducible, then $J$ must be contained in an embedded three-ball.  In this case, $E(K')_0(J) = E(K') \# S^3_0(J)$ and hence $J$ is unknotted in the embedded three-ball.  Again, $J$ is trivial in $E(K')$ and we are done.  
\end{proof}

Recently, Conway and Tosun \cite{ConwayTosun} showed that the boundary of a non-trivial Mazur manifold is not an L-space.  Ni has pointed out that an alternate proof follows from \cite{Ni:nonseparating}.  We now show how Theorem~\ref{thm:main} gives another alternate proof of this fact.  (The second author and Pinz\'on-Caicedo have also proved the analogous result in instanton Floer homology.)    
\begin{corollary}[{\cite[Theorem 1]{ConwayTosun}}]
Let $Y \neq S^3$ be a homology sphere bounding a Mazur manifold.  Then $Y$ is not an L-space.  
\end{corollary}
\begin{proof}
Suppose that $Y$ is an L-space homology sphere which bounds a Mazur manifold.  Then, there exists a knot $K$ in $Y$ such that $Y_0(K) = S^2 \times S^1$.  Since $\dim \widehat{HF}(Y) = \dim \widehat{HF}(S^3)$, Theorem~\ref{thm:main} implies that $K$ is unknotted.  Therefore, $Y_0(K) = Y \# S^2 \times S^1$, and we see that $Y = S^3$.   
\end{proof}

Finally, we also present a symplectic analogue of Corollary~\ref{cor:trivial}.  This was explained to the authors by Steven Sivek.    
\begin{corollary}\label{cor:stein}
Let $Y$ be a rational homology sphere.  Let $W$ be a Stein cobordism from $(Y,\xi)$ to $(Y,\xi')$ comprised of attaching single Weinstein 1- and 2-handles.  If $\xi'$ is tight, then $W$ is deformation equivalent to the (compact) symplectization of $(Y,\xi)$ and hence $\xi$ and $\xi'$ are contactomorphic contact structures.  
\end{corollary}
\begin{proof}
Consider the $(tb-1)$-framed 2-handle attachment to a Legendrian $\mathcal{K}$ in $(Y \# S^2 \times S^1, \xi \# \xi_{std})$ which results in $(Y,\xi')$.  By reversing this picture, we see that there is a Legendrian knot $\mathcal{K}'$ in $(Y,\xi')$ with a contact $+1$-surgery to $(Y \# S^2 \times S^1, \xi \# \xi_{std})$ by \cite[Proposition 8]{DingGeiges}.  Note that $\mathcal{K}'$ must be nullhomologous and the framing of the surgery must be the Seifert framing in order to add a $\mathbb{Z}$-summand to $H_1$.  Now, by Theorem~\ref{thm:main}, $\mathcal{K}'$ is unknotted topologically.  Since $+1$-contact surgery means that the topological framing is one more than $tb$, we see that $tb = -1$.  Because $\xi'$ is tight, this implies $r = 0$ by \cite[Theorem 1.6]{EliashbergFraser}, and all such Legendrian unknots are Legendrian isotopic by \cite[Theorem 1.5]{EliashbergFraser}.  

This implies that all Stein cobordisms from $(Y,\xi'')$ to $(Y,\xi')$ built out of single Weinstein 1- and 2-handles are equivalent, regardless of $\xi''$.  However, we can produce such a cobordism by using a cancelling Weinstein 1- and 2-handle pair, i.e., the trivial cobordism.  
\end{proof}

Finally, we give a new obstruction to a homology sphere admitting an $S^2 \times S^1$ surgery (and hence bounding a Mazur manifold).    
\begin{proposition}\label{prop:red-1}
Let $K$ be a knot in a homology sphere with $\HF_{\textup{red},i}(Y) = \F$ for some $i$.  Then $Y_0(K) \neq S^2 \times S^1$.  
\end{proposition}

\begin{remark}
It is easy to see that if a nullhomologous knot in a rational homology sphere admits a 0-surgery with an $S^2 \times  S^1$ summand, then its Alexander polynomial is trivial (i.e., constant).  We leave it as a fun exercise for the reader to deduce this fact using Heegaard Floer homology after reading the arguments in this paper.
\end{remark}

\section*{Organization} The key idea in the proof of Theorem \ref{thm:main} comes from the special property of the twisted Heegaard Floer homology of three-manifolds with non-separating $S^2$'s.  (This has been used in \cite{Ni:HFfibred} and \cite{Ni:nonseparating}; see also \cite{HeddenYi:smallHF}, \cite{HeddenYi:KhovanovUnlink}, and \cite{AlishahiLipshitz}.)  In the next section, we review the mapping cone formula in Heegaard Floer homology, with extra attention to twisted coefficients, and prove Theorem~\ref{thm:main}. Lastly, we prove Proposition~\ref{prop:red-1}.

\section*{Acknowledgements} The authors thank Maggie Miller for helpful conversations and Steven Sivek for describing the proof of Corollary~\ref{cor:stein}.  They also thank Matt Hedden and Yi Ni for helpful comments on an earlier draft of this paper.

\section{The mapping cone}\label{sec:cone}
We assume that the reader is familiar with the knot Floer chain complex of a knot $\CFKi$, and the mapping cone formula for the Heegaard Floer homology of $0$-surgery along a nullhomologous knot $K$ in a rational homology sphere $Y$ \cite[Section 4.8]{OS:integer}. We briefly recall the formula here, primarily to establish notation.  Let $\mathfrak{t}$ denote a Spin$^c$ structure on $Y$. As a vector space, we have that $C = CFK^\infty(Y,K,\mft)$ decomposes as a direct sum $C = \bigoplus_{i,j \in \Z} C(i,j)$. For any set $X \subset \Z^2$ which is convex with respect to the product partial order on $\Z^2$ (i.e., if $a < b < c$ and $a, c \in X$, then $b \in X$), let $CX = \bigoplus_{(i,j) \in X} C(i,j)$ which is naturally a subquotient complex of $C$.

Let $B^+_s$ (respectively $\widehat{B}_s$) denote $C\{i \geq 0\}$ (respectively $C\{i = 0\}$), and $A^+_s$ (respectively $\widehat{A}_s$) denote $C\{\max(i,j-s) \geq 0\}$ (respectively $C\{\max(i,j-s)=0\}$).  Recall the maps $v^+_s, h^+_s: A^+_s \to  B^+$ and $\widehat{v}_s, \widehat{h}_s: \widehat{A}_s \to \widehat{B}$.  The main fact that we will need is that $\widehat{v}_s$ factors through $\widehat{v}_{s'}$ for $s' \geq s$. 

Let $\widehat{F} \subset Y_0(K)$ denote the surface obtained by capping off an oriented Seifert surface $F$ for $K$. As usual, we let $\mft_s$ denote the Spin$^c$ structure on $Y_0(K)$ which satisfies $\langle c_1(\mft_s), [\widehat{F}]\rangle = 2s$ and such that $\mft_s$ extends $\mft$ over the 0-framed 2-handle cobordism from $Y$ to $Y_0(K)$. 
In what follows, let $\circ$ denote either $+$ or $\; \widehat{\;}$.  

\begin{theorem}[{\cite[Theorem 9.19]{OS:properties}, see also \cite[Section 4.8]{OS:integer}}]\label{thm:cone}
Let $Y$ be a rational homology sphere and $K \subset Y$ a null-homologous knot. With notation as above,
\[
	HF^\circ(Y_0(K), \mft_s) \cong H_*(\Cone (v_s^\circ + h_s^\circ)).
\]
\end{theorem}
There is also a version of Theorem \ref{thm:cone} with twisted coefficients, as in \cite[Section 8]{OS:properties}; see also \cite[Section 2]{JabukaMark} and \cite[Section 2]{LevineRuberman}. Let $T$ be a generator of $H^1(Y_0(K); \Z)$. Consider the map
\[
	v^\circ_s + T h^\circ_s \co A^\circ_s \otimes_\F \F[T, T^{-1}] \to B^\circ_s \otimes_\F \F[T, T^{-1}] 
\]
We have the following mapping cone formula with twisted coefficients.  We write $HF^\circ(Y_0(K), \mft_s; \F[T,T^{-1}])$ to denote the Heegaard Floer homology with totally twisted coefficients.  We will also write $HF^\circ(Y_0(K), \mft_s;\F[[T,T^{-1}])$ to be the homology of the chain complex obtained by tensoring the twisted Heegaard Floer chain complex $CF^\circ(Y_0(K), \mft_s; \F[T,T^{-1}])$ with $\F[[T,T^{-1}]$ over $\F[T,T^{-1}]$.   

\begin{theorem}[{\cite[Theorem 9.23]{OS:properties}, see also \cite[Theorem 2.3]{LevineRuberman}}]\label{thm:twistedcone}
Let $Y$ be a rational homology sphere and $K \subset Y$ a nullhomologous knot. With notation as above,
\[
	HF^\circ(Y_0(K), \mft_s; \F[T, T^{-1}]) \cong H_*(\Cone (v^\circ_s + T h^\circ_s)).	
\]
\end{theorem}

We will be interested in the following consequence of the preceding theorem.

\begin{corollary}\label{cor:NovikovT}
Let $Y$ be a rational homology sphere and $K \subset Y$ nullhomologous. Then $HF^\circ(Y_0(K), \mft_s; \F[[T, T^{-1}])$ is isomorphic to the homology of the cone of
\[ v^\circ_s + Th^\circ_s \co A^\circ_s \otimes_\F \F[[T, T^{-1}] \to B^\circ_s \otimes_\F \F[[T, T^{-1}]. 
\]
\end{corollary}
\begin{proof}
The result follows from Theorem \ref{thm:twistedcone} and the fact that $\F[[T, T^{-1}]$ is flat over $\F[T, T^{-1}]$.
\end{proof}

We recall one key property of the Heegaard Floer homology of three-manifolds with non-separating two-spheres.  If $M$ is a three-manifold which contains a non-separating two-sphere $S$, then $HF^\circ(M;\F[[T,T^{-1}]) = 0$, where $T$ denotes a generator of $H^1$ of the $S^2 \times S^1$ summand \cite[Lemma 2.1]{Ni:HFfibred}. Further, if $\mfs$ is a Spin$^c$ structure on $M$ such that $\langle c_1(\mfs), [S] \rangle = 0$, then $HF^\circ(M,\mfs) = 0$ \cite[Theorem 1.4]{OS:properties}.  With this, we analyze the mapping cone formula for knots which surger to three-manifolds with non-separating two-spheres.

\begin{proposition}\label{prop:0surgery}
Let $Y$ be a rational homology sphere and $K \subset Y$ a nullhomologous knot.  Suppose $Y_0(K) = N \# S^2 \times S^1$.  Let $\circ = +$ or $\; \widehat{\;}$.  Then $v_{s,*}^\circ + h_{s,*}^\circ: H_*(A^\circ_s) \to HF^\circ(Y,\mft)$ is an isomorphism for all $s \neq 0$.  Further, $v_{s,*}^\circ + T h_{s,*}^\circ: H_*(A^\circ_s)\otimes_\F \F[[T,T^{-1}] \to HF^\circ(Y, \mft)\otimes_\F \F[[T,T^{-1}]$ is an isomorphism for all $s$.  In particular, $\dim H_*(\widehat{A}_s) = \dim \HFh(Y, \mft)$ for all $s$.
\end{proposition} 

\begin{proof}
The first claim follows from Theorem~\ref{thm:cone} and that $Y_0(K)$ contains a non-separating two-sphere.  

Now, for the second claim, fix $\mft$ in Spin$^c(Y)$.  Let $\mft'$ denote the Spin$^c$ structure on $N$ which is cobordant to $\mft$ under the homology cobordism from $Y$ to $N$ obtained by attaching a 3-handle to the trace of 0-surgery on $K$.  
Since $Y_0(K) = N \# S^2 \times S^1$, we have that $HF^+(Y_0(K), \mft_s; \F[[T, T^{-1}])=0$. By Corollary \ref{cor:NovikovT}, we have that 
\[ HF^+(Y_0(K), \mft_s; \F[[T, T^{-1}]) \cong H_*(\Cone (v^+_s + Th^+_s)\otimes_{\F[T, T^{-1}]} \F[[T, T^{-1}]). \]
Hence
\[ (v^+_s + Th^+_s)_* \co H_*(A^+_s \otimes_\F \F[[T, T^{-1}]) \to H_*(B^+_s \otimes_\F \F[[T, T^{-1}] ) \]
is an isomorphism of $\F[[T, T^{-1}]$-modules.  
The analogous result for the hat flavor follows immediately.  
\end{proof}

\begin{proof}[Proof of Theorem \ref{thm:main}]
As before, fix $\mft$ in Spin$^c(Y)$.  Let $\mft'$ denote the Spin$^c$ structure on $N$ which is cobordant to $\mft$ under the homology cobordism from $Y$ to $N$ obtained by attaching a 3-handle to the trace of 0-surgery on $K$. Suppose that $\dim_\F \HFh(Y, \mft) \leq \dim_\F \HFh(N, \mft')$. We then have
\begin{align*}
	2 \dim_\F \HFh(N, \mft') &= \dim_\F (\HFh(N \# S^2 \times S^1, \mft' \# \mfs_0)  \\
		&= \dim_\F (H_*(\Cone(\widehat{v}_0 + \widehat{h}_0)) \\
 		&= \dim_\F H_*(\widehat{A}_0) + \dim_\F \HFh(Y,\mft) - 2 \rk (\widehat{v}_{0,*} + \widehat{h}_{0,*}) \\
		&= 2 \dim_\F \HFh(Y, \mft) - 2 \rk (\widehat{v}_{0,*} + \widehat{h}_{0,*}) \\
		&\leq 2 \dim_\F \HFh(N, \mft') - 2 \rk (\widehat{v}_{0,*} + \widehat{h}_{0,*}),
\end{align*}
where the first equality follows from the K\"unneth formula, the second from Theorem \ref{thm:twistedcone}, the third from rank-nullity (and the fact that we are working over a field), the fourth from Propsition \ref{prop:0surgery}, and the final inequality by hypothesis.
Hence, we see that $\widehat{v}_{0,*} =\widehat{h}_{0,*}$. Therefore,
\[ (1+ T)\widehat{v}_{0,*} \co H_*(\widehat{A}_0 \otimes_\F \F[[T, T^{-1}]) \to H_*(\widehat{B}_0 \otimes_\F \F[[T, T^{-1}] ) \]
is an isomorphism. This implies that $\widehat{v}_{0,*}$ is an isomorphism.

We now consider the case $s > 0$. As mentioned above, $\widehat{v}_{0,*}$ factors through $\widehat{v}_{s,*}$. In particular, since $\widehat{v}_{0,*}$ is an isomorphism, we have that $\widehat{v}_{s,*}$ is surjective. Therefore, it suffices to show that $\dim H_*(\widehat{A}_s) = \dim \HFh(Y, \mft)$.  This again follows from Proposition~\ref{prop:0surgery}. Since $\widehat{v}_{s,*}$ is an isomorphism if and only if $v^+_{s,*}$ is an isomorphism, it follows from \cite[Theorem 1.2]{OS:genus} (which holds for nullhomologous knots in arbitrary rational homology spheres) and \cite[Proof of Lemma 8.1]{OS:rational} that
\[
g(K) = \min \{ s \mid \widehat{v}_{i,*} \text{ is an isomorphism for all }i\geq s, \ \mft \in \text{Spin}^c(Y) \} \leq 0,
\]
which gives the desired result.  
\end{proof}

\begin{proof}[Proof of Proposition \ref{prop:red-1}]
This is very similar to the proof of Theorem~\ref{thm:main}.  After a possible orientation reversal, we may assume that $HF_{\textup{red},i}(Y) = \F$ and $i$ is odd.  By Proposition~\ref{prop:0surgery}, $H_i(A^+_{0}) = \F$, and $v^+_{0,*} + T h^+_{0,*} : H_i(A^+_0)  \otimes_\F \F[[T, T^{-1}] \to H_i(B^+_0) \otimes_\F \F[[T, T^{-1}]$ is an isomorphism.  (Here, we are using the fact that $v^+_0$ and $h^+_0$ are homogeneous of the same grading shift.  This is not true for $s \neq 0$.)  Restricted to this grading, this latter map can be written as $v^+_0 + T h^+_0 \co \F[[T,T^{-1}] \to \F[[T,T^{-1}]$.  It follows that either $v^+_0$ or $h^+_0$ must be non-zero as a map from $H_i(A^+_0) = \F$ to $H_i(B^+_0) = \F$.  By conjugation invariance \cite[Theorem 3.6]{OS:smooth4}, we have that $v^+_0$ is non-zero if and only if $h^+_0$ is non-zero, and so they must be equal.  Therefore, $v^+_{0,*}  = h^+_{0,*}$ as maps from $H_i(A^+_0)$ to $H_i(B^+_0)$, and we see that the kernel of $v^+_{0,*} + h^+_{0,*}$ contains an $\F$ in grading $i$, which is odd.  

Consider the homology of the cone of $v^+_{0,*} + h^+_{0,*} \co H_*(A^+_0) \to H_*(B^+_0)$.  This has two towers: one from the kernel of $v^+_{0,*} + h^+_{0,*}$ and one from the cokernel.  We also know there is an additional generator in the kernel of $v^+_{0,*} + h^+_{0,*}$ in degree $i$; this is in opposite parity of the tower found in this kernel.  This contradicts the fact that the homology of the cone of $v^+_{0,*} + h^+_{0,*} \co H_*(A^+_0) \to H_*(B^+_0)$ is the associated graded object for a 2-step filtration on $HF^+(S^2 \times S^1) \cong \mathcal{T}^+ \oplus \mathcal{T}^+$ by the mapping cone formula.  Indeed, some $U$-torsion elements in the cokernel of $v^+_{0,*} + h^+_{0,*}$ could correspond to elements in a tower of $HF^+(S^2 \times S^1)$, but elements in the kernel cannot by $U$-equivariance.    
\end{proof}

\bibliographystyle{alpha}
\bibliography{references}

\begin{thebibliography}{DLVVW19}

\bibitem[AL19]{AlishahiLipshitz}
Akram Alishahi and Robert Lipshitz.
\newblock Bordered {F}loer homology and incompressible surfaces.
\newblock {\em Ann. Inst. Fourier (Grenoble)}, 69(4):1525--1573, 2019.

\bibitem[CT18]{ConwayTosun}
James Conway and B\"ulent Tosun.
\newblock Mazur-type manifolds with {L}-space boundaries.
\newblock ar{X}iv:1807.08880, 2018.

\bibitem[DG01]{DingGeiges}
Fan Ding and Hansj\"{o}rg Geiges.
\newblock Symplectic fillability of tight contact structures on torus bundles.
\newblock {\em Algebr. Geom. Topol.}, 1:153--172, 2001.

\bibitem[DLVVW19]{DLVVW}
Aliakbar Daemi, Tye Lidman, David~Shea Vela-Vick, and C.-M.~Michael Wong.
\newblock Ribbon homology cobordisms.
\newblock ar{X}iv:1904.09721, 2019.

\bibitem[EF09]{EliashbergFraser}
Yakov Eliashberg and Maia Fraser.
\newblock Topologically trivial {L}egendrian knots.
\newblock {\em J. Symplectic Geom.}, 7(2):77--127, 2009.

\bibitem[Gab87a]{Gabai:propR}
David Gabai.
\newblock Foliations and the topology of {$3$}-manifolds. {III}.
\newblock {\em J. Differential Geom.}, 26(3):479--536, 1987.

\bibitem[Gab87b]{Gabai:superadditive}
David Gabai.
\newblock Genus is superadditive under band connected sum.
\newblock {\em Topology}, 26(2):209--210, 1987.

\bibitem[GL89]{GordonLuecke}
C.~McA. Gordon and J.~Luecke.
\newblock Knots are determined by their complements.
\newblock {\em J. Amer. Math. Soc.}, 2(2):371--415, 1989.

\bibitem[Gor81]{Gordon:ribbon}
C.~McA. Gordon.
\newblock Ribbon concordance of knots in the {$3$}-sphere.
\newblock {\em Math. Ann.}, 257(2):157--170, 1981.

\bibitem[Gor98]{Gordon:fillingsurvey}
C.~McA. Gordon.
\newblock Dehn filling: a survey.
\newblock In {\em Knot theory ({W}arsaw, 1995)}, volume~42 of {\em Banach
  Center Publ.}, pages 129--144. Polish Acad. Sci. Inst. Math., Warsaw, 1998.

\bibitem[HN10]{HeddenYi:smallHF}
Matthew Hedden and Yi~Ni.
\newblock Manifolds with small {H}eegaard {F}loer ranks.
\newblock {\em Geom. Topol.}, 14(3):1479--1501, 2010.

\bibitem[HN13]{HeddenYi:KhovanovUnlink}
Matthew Hedden and Yi~Ni.
\newblock Khovanov module and the detection of unlinks.
\newblock {\em Geom. Topol.}, 17(5):3027--3076, 2013.

\bibitem[JM08]{JabukaMark}
Stanislav Jabuka and Thomas~E. Mark.
\newblock Product formulae for {O}zsv\'{a}th-{S}zab\'{o} 4-manifold invariants.
\newblock {\em Geom. Topol.}, 12(3):1557--1651, 2008.

\bibitem[Lin19]{LinStein}
Francesco Lin.
\newblock Indefinite {S}tein fillings and ${Pin(2)}$-monopole {F}loer homology,
  2019.
\newblock arXiv:1907.07566.

\bibitem[LR19]{LevineRuberman}
Adam~Simon Levine and Daniel Ruberman.
\newblock Heegaard {F}loer invariants in codimension one.
\newblock {\em Trans. Amer. Math. Soc.}, 371(5):3049--3081, 2019.

\bibitem[Ni09]{Ni:HFfibred}
Yi~Ni.
\newblock Heegaard {F}loer homology and fibred 3-manifolds.
\newblock {\em Amer. J. Math.}, 131(4):1047--1063, 2009.

\bibitem[Ni13]{Ni:nonseparating}
Yi~Ni.
\newblock Nonseparating spheres and twisted {H}eegaard {F}loer homology.
\newblock {\em Algebr. Geom. Topol.}, 13(2):1143--1159, 2013.

\bibitem[OS04a]{OS:genus}
Peter Ozsv\'{a}th and Zolt\'{a}n Szab\'{o}.
\newblock Holomorphic disks and genus bounds.
\newblock {\em Geom. Topol.}, 8:311--334, 2004.

\bibitem[OS04b]{OS:knots}
Peter Ozsv\'{a}th and Zolt\'{a}n Szab\'{o}.
\newblock Holomorphic disks and knot invariants.
\newblock {\em Adv. Math.}, 186(1):58--116, 2004.

\bibitem[OS04c]{OS:properties}
Peter Ozsv\'{a}th and Zolt\'{a}n Szab\'{o}.
\newblock Holomorphic disks and three-manifold invariants: properties and
  applications.
\newblock {\em Ann. of Math. (2)}, 159(3):1159--1245, 2004.

\bibitem[OS06]{OS:smooth4}
Peter Ozsv\'{a}th and Zolt\'{a}n Szab\'{o}.
\newblock Holomorphic triangles and invariants for smooth four-manifolds.
\newblock {\em Adv. Math.}, 202(2):326--400, 2006.

\bibitem[OS08]{OS:integer}
Peter Ozsv\'{a}th and Zolt\'{a}n Szab\'{o}.
\newblock Knot {F}loer homology and integer surgeries.
\newblock {\em Algebr. Geom. Topol.}, 8(1):101--153, 2008.

\bibitem[OS11]{OS:rational}
Peter Ozsv\'{a}th and Zolt\'{a}n Szab\'{o}.
\newblock Knot {F}loer homology and rational surgeries.
\newblock {\em Algebr. Geom. Topol.}, 11(1):1--68, 2011.

\end{thebibliography}

\end{document}